\documentclass{amsart}

\usepackage{amsfonts}
\usepackage{amsmath}
\usepackage{amssymb}
\usepackage{setspace}
\usepackage{graphicx}
\usepackage{amscd,amssymb,amsthm}

\usepackage{hyperref}
\usepackage{graphicx}
\usepackage{amsmath, amsthm, amscd, amsfonts, amssymb, graphicx, color}
 \usepackage{url}
\usepackage{enumerate}
 \makeatletter
\let\reftagform@=\tagform@
\def\tagform@#1{\maketag@@@{(\ignorespaces\textcolor{blue}{#1}\unskip\@@italiccorr)}}
\renewcommand{\eqref}[1]{\textup{\reftagform@{\ref{#1}}}}
\makeatother
\usepackage{hyperref}
\hypersetup{colorlinks=true, linkcolor=red, anchorcolor=green,
citecolor=cyan, urlcolor=red, filecolor=magenta, pdftoolbar=true}

\usepackage{epsfig}        
\usepackage{epic,eepic}       

\newtheorem{theorem}{Theorem}
\theoremstyle{plain}

\newtheorem{corollary}{Corollary}

\numberwithin{equation}{section}

\begin{document}

\title[On Alzer's inequality]{On Alzer's inequality}
\author[M.W. Alomari]{Mohammad W. Alomari}
\address{Department of Mathematics, Faculty of Science and
Information Technology, Irbid National University, P.O. Box 2600,
Irbid, P.C. 21110, Jordan.} \email{mwomath@gmail.com}


\date{\today}
\subjclass[2000]{26D15}

\keywords{Wirtinger inequality, Alzer inequality, Trapezoid
inequality}

\begin{abstract}
Extensions and generalizations of Alzer's inequality; which is of
Wirtinger type are proved.   As applications, sharp trapezoid type
inequality and sharp bound for the geometric mean are deduced.
\end{abstract}

\maketitle

\section{Introduction}

In Fourier analysis, the theory of inequalities plays an important
and useful role in almost all branches of its analyzes. Early of
the last century, several famous inequalities have been used in
the theory of Fourier series, Fourier integrals and Fourier
transform. The inequalities of Bessel, Blaschke, Wirtinger,
Beesack and others, are used at large in convergence and
estimations of such series and integrals.

In \cite{Blaschke}, Wirtinger proved the following inequality
regarding square integrable functions:

\begin{theorem}
\label{thm1}Let $f$ be a real valued function with period $2\pi$
and $\int_0^{2\pi}f\left({x}\right)dx=0$. If $f' \in L^2[0,2\pi]$,
then
\begin{align}
\label{eq1.1}\int_0^{2\pi } {f^2 \left( x \right)dx}  \le
\int_0^{2\pi } {f'^2 \left( x \right)dx},
\end{align}
with equality if and only if $f(x)=A\cos x + B \sin x$, $A,B\in
\mathbb{R}$.
\end{theorem}
Various generalizations, counterparts and refinements were
considered in \cite{Alzer}--\cite{Gradimir} and the references
therein.

In \cite{Alzer}, Alzer introduced a Wirtinger like inequality for
continuously differentiable periodic functions, which reads:
\begin{theorem}
\label{thm2}If $f$ is a real valued continuously differentiable
function with period $2\pi$ and
$\int_0^{2\pi}f\left({x}\right)dx=0$, then
\begin{align}
\label{eq1.2}\frac{6}{\pi }\mathop {\max }\limits_{0 \le x \le
2\pi } f^2 \left( x \right) \le \int_0^{2\pi } {f'^2 \left( x
\right)dx}.
\end{align}
Equality holds if and only if $f\left( x \right) = c\left[
{3\left( {\frac{{x - \pi }}{\pi }} \right)^2  - 1} \right]$, $0
\le x \le 2\pi$ and $c\in \mathbb{R}$.
\end{theorem}

The aim of this work is to extend and generalize Alzer inequality
(\ref{eq1.2}), by relaxing the assumptions:  continuity of $f'$,
 periodicity and the interval involved for various kind of
 functions.

\section{The Results}\label{sec2}

The version of Alzer inequality for convex functions may be stated
as follows:
\begin{theorem}
\label{thm3}Let $f:I\subseteq \mathbb{R}\rightarrow \mathbb{R}$ be
a convex mapping on $I^{\circ },$ the interior of the interval
$I$, where $a,b\in I^{\circ }$ with $a<b$, such that $f^{\prime
}\in L[a,b]$. If $f\left({a}\right)f\left({b}\right)>0$ and
$\int_a^b{f\left({t}\right)dt}=0$, then the inequality
\begin{align}
\label{eq2.1} f\left({a}\right)f\left({b}\right) \le
\frac{b-a}{12}\cdot\int_a^b {f'^2 \left( x \right)dx},
\end{align}
holds. The constant `$\frac{b-a}{12}$' is the best possible, in
the sense that it cannot be replaced by a smaller constant.
\end{theorem}

\begin{proof}
Assume that $f$ attains its maximum value at $x_0 \in [a,b]$ and
let $\mathop {\max }\limits_{a \le x \le b} f\left( x \right) =
f\left( {x_0 } \right)$, for some $a \le x_0 \le b$, then
\begin{align}
0&\le \int_a^b {\left[ {\frac{{f'\left( x \right)}}{{f\left( {x_0
} \right)}} - \frac{12}{\left( {b - a } \right)^2}\cdot\left( {x -
\frac{a+b}{2} } \right)} \right]^2 dx}
\nonumber\\
&=\int_a^b {\frac{{f'^2 \left( x \right)}}{{f^2 \left( {x_0 }
\right)}}dx}  - \frac{{24 }}{{\left( {b - a} \right)^2 f\left(
{x_0 } \right)}}\int_a^b {\left( {x - \frac{a+b}{2} }
\right)f'\left( x \right)dx}\label{eq2.2}
\\
&\qquad+ \frac{{144 }}{{\left( {b - a} \right)^4 }}\int_a^b
{\left( {x - \frac{a+b}{2} } \right)^2 dx}.\nonumber
\end{align}
Observing that
\begin{align*}
\int_a^b {\left( {x - \frac{a+b}{2} } \right)f'\left( x
\right)dx}&= \frac{b-a}{2}\cdot
\left[{f\left({a}\right)+f\left({b}\right)}\right]-\int_a^b
{f\left( x \right)dx},
\end{align*}
taking in account that $\int_a^b {f\left( x \right)dx}=0$.
Substituting in (\ref{eq2.2}), we get
\begin{align*}
0&\le \int_a^b {\left[ {\frac{{f'\left( x \right)}}{{f\left( {x_0
} \right)}} - \frac{12}{\left( {b - a } \right)^2}\cdot\left( {x -
\frac{a+b}{2} } \right)} \right]^2 dx}
\\
&=\int_a^b {\frac{{f'^2 \left( x \right)}}{{f^2 \left( {x_0 }
\right)}}dx}   - \frac{{12}}{\left( {b - a}
\right)f\left({x_0}\right)}\left[{f\left({a}\right)+f\left({b}\right)}\right]+
\frac{{12}}{\left( {b - a} \right)}
\end{align*}
which gives that
\begin{align*}
\left\{{ f\left({a}\right)+f\left({b}\right)-f\left({x_0}\right)
}\right\}\cdot \mathop {\max }\limits_{a \le x \le b} f\left( x
\right) \le \frac{b-a}{12}\cdot\int_a^b {f'^2 \left( x \right)dx}.
\end{align*}
Finally, since $f$ is convex then $f$ attains its maximum at the
endpoints `$a$' or `$b$', so if $\mathop {\max }\limits_{a \le x
\le b} f\left( x \right)=f\left( b \right)=f\left( x_0 \right) $,
thus we have
\begin{align}
\label{eq2.3}f\left({a}\right)\cdot \mathop {\max }\limits_{a \le
x \le b} f\left( x \right) \le \frac{b-a}{12}\cdot\int_a^b {f'^2
\left( x \right)dx},
\end{align}
and if $\mathop {\max }\limits_{a \le x \le b} f\left( x
\right)=f\left( a \right)=f\left( x_0 \right) $, we have
\begin{align}
\label{eq2.4}f\left({b}\right)\cdot \mathop {\max }\limits_{a \le
x \le b} f\left( x \right) \le \frac{b-a}{12}\cdot\int_a^b {f'^2
\left( x \right)dx}.
\end{align}
So that the both inequalities  (\ref{eq2.3}) and (\ref{eq2.4}),
can be read as
\begin{align*}
f\left({a}\right)f\left({b}\right) \le \frac{b-a}{12}\cdot\int_a^b
{f'^2 \left( x \right)dx},
\end{align*}
and thus the proof of (\ref{eq2.1}) is established. To prove the
sharpness of (\ref{eq2.1}), let (\ref{eq2.1}) holds with another
constant $C>0$,
\begin{align}
\label{eq2.5}f\left({a}\right)f\left({b}\right) \le C\cdot\int_a^b
{f'^2 \left( x \right)dx}.
\end{align}
Define the function $f:[0,1] \to \mathbb{R}$ defined by $f\left( x
\right)= 6x^2- 6x+1$, for all $x \in [0,1]$. Clearly, $f$ is
convex for all $x\in [0,1]$. Moreover, we have $f\left( 0
\right)=f\left( 1 \right)=1$, and $\int_0^1{f'^2(x)dx}=12$. Making
use of (\ref{eq2.5}), we have $C\ge\frac{1}{12}$, and this proves
the best possibility of $\frac{1}{12}$, which completes the proof.
\end{proof}

The following inequality for monotonic mappings holds.
\begin{theorem}
\label{thm4}Let $f:I\subseteq \mathbb{R}\rightarrow \mathbb{R}$ be
an increasing function on $I^{\circ },$ the interior of the
interval $I$, where $a,b\in I^{\circ }$ with $a<b$, such that
$f^{\prime }\in L[a,b]$. If $\int_a^b{f\left({t}\right)dt}=0$,
then the inequality
\begin{align}
\label{eq2.6} \left[{ 2f\left({a}\right)- f\left({b}\right)
}\right]\cdot f\left({b}\right) \le \frac{b-a}{12}\cdot\int_a^b
{f'^2 \left( x \right)dx},
\end{align}
holds. The constant `$\frac{b-a}{12}$' is the best possible.
\end{theorem}
\begin{proof}
Repeating the steps in the proof of Theorem \ref{thm3}, since $f$
is bounded and monotonically increasing on $[a,b]$, then $f(a)\le
f(t)$ for all $t\in [a,b]$, therefore
\begin{align*}
0&\le \int_a^b {\left[ {\frac{{f'\left( x \right)}}{{f\left( {x_0
} \right)}} - \frac{12}{\left( {b - a } \right)^2}\cdot\left( {x -
\frac{a+b}{2} } \right)} \right]^2 dx}
\\
&=\int_a^b {\frac{{f'^2 \left( x \right)}}{{f^2 \left( {x_0 }
\right)}}dx}  - \frac{{12}}{\left( {b - a}
\right)f\left({x_0}\right)}\left[{f\left({a}\right)+f\left({b}\right)}\right]
+ \frac{{12}}{\left( {b - a} \right)}
\\
&\le\int_a^b {\frac{{f'^2 \left( x \right)}}{{f^2 \left( {x_0 }
\right)}}dx} - \frac{{24}}{\left( {b - a}
\right)f\left({x_0}\right)}f\left({a}\right) + \frac{{12}}{\left(
{b - a} \right)}
\end{align*}
which gives that
\begin{align*}
\left[{ 2f\left({a}\right)- f\left({b}\right) }\right]\cdot
f\left({b}\right) \le \frac{b-a}{12}\cdot\int_a^b {f'^2 \left( x
\right)dx},
\end{align*}
which proves the inequality (\ref{eq2.6}). The sharpness holds
with the function $f \left( x \right) = 4c^2 \cdot x^3  + 12c\cdot
x - c^2 - 6c$, for all $x\in [0,1]$, where $c = \frac{{ 10 +
2\sqrt {835} }}{{27}}$.
\end{proof}

\begin{corollary}
\label{cor1}Let $f:I\subseteq \mathbb{R}\rightarrow \mathbb{R}$ be
a bounded decreasing function on $I^{\circ },$ the interior of the
interval $I$, where $a,b\in I^{\circ }$ with $a<b$, such that
$f^{\prime }\in L[a,b]$. If $\int_a^b{f\left({t}\right)dt}=0$,
then the inequality
\begin{align}
\label{eq2.7} \left[{ 2f\left({b}\right)- f\left({a}\right)
}\right]\cdot f\left({a}\right) \le \frac{b-a}{12}\cdot\int_a^b
{f'^2 \left( x \right)dx},
\end{align}
holds. The constant `$\frac{b-a}{12}$' is the best possible.
\end{corollary}

\begin{proof}
The proof is similar the proof of Theorem \ref{thm4}.
\end{proof}

In general, we may generalize and extend Alzer inequality
(\ref{eq1.2}) as follows:
\begin{theorem}
\label{thm5}Let $f:I\subseteq \mathbb{R}\rightarrow \mathbb{R}$ be
an absolutely continuous mapping on $I^{\circ },$ the interior of
the interval $I$, where $a,b\in I^{\circ }$ with $a<b$, such that
$f^{\prime }\in L[a,b]$. If  $f\left( a \right)=\mathop {\max
}\limits_{a \le x \le b} f\left( x \right)= f\left( b \right)$ and
$\int_a^b{f\left({t}\right)dt}=0$, then the inequality
\begin{align}
\label{eq2.8} \mathop {\max }\limits_{a \le x \le b} f^2\left( x
\right) \le \frac{b-a}{12}\cdot \int_a^b {f'^2 \left( x
\right)dx},
\end{align}
holds. The constant `$\frac{b-a}{12}$' is the best possible.
\end{theorem}

\begin{proof}
Given the assumptions. Assume that $f$ attains its maximum value
at $x_0 \in [a,b]$ and let $\mathop {\max }\limits_{a \le x \le b}
f\left( x \right) = f\left( {x_0 } \right)$, for some $a \le x_0
\le b$, then
\begin{align}
0&\le \int_a^b {\left[ {\frac{{f'\left( x \right)}}{{f\left( {x_0
} \right)}} - \frac{12}{\left( {b - a } \right)^2}\cdot\left( {x -
\frac{a+b}{2} } \right)} \right]^2 dx}
\nonumber\\
&=\int_a^b {\frac{{f'^2 \left( x \right)}}{{f^2 \left( {x_0 }
\right)}}dx}  - \frac{{24}}{{\left( {b - a} \right)^2 f\left( {x_0
} \right)}}\int_a^b {\left( {x - \frac{a+b}{2} } \right)f'\left( x
\right)dx}\label{eq2.9}
\\
&\qquad+ \frac{{144 }}{{\left( {b - a} \right)^4 }}\int_a^b
{\left( {x - \frac{a+b}{2} } \right)^2 dx}.\nonumber
\end{align}
Since $f\left( a \right)=\mathop {\max }\limits_{a \le x \le b}
f\left( x \right)= f\left( b \right)$, we have
\begin{align*}
\int_a^b {\left( {x - \frac{a+b}{2} } \right)f'\left( x
\right)dx}&= \frac{b-a}{2}\cdot
\left[{f\left({a}\right)+f\left({b}\right)}\right]
=\left({b-a}\right)\cdot f\left({x_0}\right)
\end{align*}
Substituting in (\ref{eq2.9}),
\begin{align*}
0\le \int_a^b {\left[ {\frac{{f'\left( x \right)}}{{f\left( {x_0 }
\right)}} - \frac{12}{\left( {b - a } \right)^2}\cdot\left( {x -
\frac{a+b}{2} } \right)} \right]^2 dx} &=\int_a^b {\frac{{f'^2
\left( x \right)}}{{f^2 \left( {x_0 } \right)}}dx}  -
\frac{24}{b-a} + \frac{{12 }}{{b-a}}
\\
&= \int_a^b {\frac{{f'^2 \left( x \right)}}{{f^2 \left( {x_0 }
\right)}}dx}  - \frac{{12}}{b-a}
\end{align*}
which gives that
\begin{align*}
\mathop {\max }\limits_{a \le x \le b} f^2\left( x \right) \le
\frac{b-a}{12} \cdot \int_a^b {f'^2 \left( x \right)dx},
\end{align*}
and thus the proof of (\ref{eq2.9}) is established. To prove the
sharpness of (\ref{eq2.8}), let $a=0$, $b=2\pi$, then
(\ref{eq2.8}) reduces to (\ref{eq1.2}), so by considering the same
function $f$ as given in Theorem \ref{thm2}, we get the sharpness.
\end{proof}

The most extensive case holds without any additional restrictions
on $f$ is considered as follows:
\begin{theorem}
\label{thm6}Let $f:I\subseteq\mathbb{R}\rightarrow \mathbb{R}$ be
an absolutely continuous mapping on $I^{\circ },$ the interior of
the interval $I$, where $a,b\in I^{\circ }$ with $a<b$, such that
$f^{\prime }\in L[a,b]$.  Then the inequality
\begin{align}
\label{eq2.10} \left[{\frac{{2}}{{b -
a}}\cdot\mathcal{T}_{\rm{rap}}\left({f}\right) - \mathop {\max
}\limits_{a \le x \le b} f\left( x \right) }\right]\cdot\mathop
{\max }\limits_{a \le x \le b} f\left( x \right) \le
\frac{b-a}{12} \cdot \int_a^b {f'^2 \left( x \right)dx},
\end{align}
holds, where
\begin{align*}
\mathcal{T}_{\rm{rap}}\left({f}\right):=\left({b-a}\right)
\frac{f\left({a}\right)+f\left({b}\right)}{2}  -\int_a^b {f\left(
x \right)dx}.
\end{align*}
The inequality is sharp.
\end{theorem}

\begin{proof}
Repeating the steps in the proof of Theorem \ref{thm5} taking in
account that no restrictions on $f$, we have
\begin{align*}
0&\le \int_a^b {\left[ {\frac{{f'\left( x \right)}}{{f\left( {x_0
} \right)}} - \frac{12}{\left( {b - a } \right)^2}\cdot\left( {x -
\frac{a+b}{2} } \right)} \right]^2 dx}
\\
&=\int_a^b {\frac{{f'^2 \left( x \right)}}{{f^2 \left( {x_0 }
\right)}}dx}  - \frac{{24}}{{\left( {b - a} \right)^2 f\left( {x_0
} \right)}}\cdot \mathcal{T}_{\rm{rap}}\left({f}\right)+ \frac{{12
}}{{\left( {b - a} \right)}}
\end{align*}
which gives that
\begin{align*}
\frac{{2\mathcal{T}_{\rm{rap}}\left({f}\right)-
\left({b-a}\right)f\left( {x_0 } \right)}}{{b - a}} \cdot\mathop
{\max }\limits_{a \le x \le b} f\left( x \right) \le
\frac{b-a}{12} \cdot \int_a^b {f'^2 \left( x \right)dx},
\end{align*}
and thus the proof of (\ref{eq2.10}) is established. The sharpness
follows with $f\left( x \right)= 6x^2- 6x+1$, for all $x \in
[0,1]$.
\end{proof}

Another generalization for $(2n)$-times differentiable functions
is considered as follows:
\begin{theorem}
\label{thm3}Let $f:I\subset \mathbb{R}\rightarrow \mathbb{R}$ be
$(2n)$-times differentiable  $(n\ge1)$ on $I^{\circ }$, the
interior of the interval $I$, where $a,b\in I^{\circ }$ with
$a<b$, such that $f^{(2n)}\in L^1[a,b]$.  If
$\int_a^b{f\left({t}\right)dt}=0$, then the inequality
\begin{align}
\label{eq2.1}\left\| {f } \right\|_\infty \le \left({\frac{b-a
}{12} }\right)^n \cdot  \left\| {f^{\left( {2n} \right)} }
\right\|_2
\end{align}
holds, where, $\left\| {f } \right\|_\infty:=\mathop {\sup
}\limits_{a \le x \le b} \left|f\left( x \right)\right|$ and
$\left\| f^{\left( {2n} \right)} \right\|^2_2=\int_a^b
{\left|f^{\left( {2n} \right)}\left( x \right) \right|^2dx}$.
\end{theorem}

\begin{proof}
Setting
\begin{align*}
\alpha  = \frac{{\left( {12} \right)^n \left( {b - a} \right)^{ -
\left( {n + \frac{1}{2}} \right)} }}{{{\rm B}^{\frac{1}{2}} \left(
{2n + 1,2n + 1} \right)}} \,\,\,\,\,\,\,\,\,\,\,\,\,\,\,\,n\in
\mathbb{N},
\end{align*}
where ${\rm B} \left( {\cdot,\cdot} \right)$ is Euler-beta
function.  Assume that $f$ attains its maximum value at $x_0 \in
[a,b]$ and let $\mathop {\sup }\limits_{a \le x \le b} f\left( x
\right) = f\left( {x_0 } \right)$, for some $a \le x_0 \le b$,
then
\begin{align*}
0&\le \int_a^b {\left[ {\frac{{f^{\left({2n}\right)}\left( x
\right)}}{{f\left( {x_0 } \right)}} - \alpha \cdot\frac{\left( {x
- a } \right)^n\left( {b - x } \right)^n}{\left( {b - a }
\right)^{2n}}} \right]^2 dx}
\\
&=\int_a^b {\frac{{\left(f^{\left({2n}\right)}\left( x
\right)\right)^2}}{{f^2 \left( {x_0 } \right)}}dx}  -
\frac{{2\alpha }}{{\left( {b - a} \right)^{2n} f\left( {x_0 }
\right)}}\int_a^b {\left( {x - a} \right)^n\left( {b - x}
\right)^nf^{\left({2n}\right)}\left( x \right)dx}
\\
&\qquad+ \frac{{\alpha^2 }}{{\left( {b - a} \right)^{4n}
}}\int_a^b {\left( {x - a} \right)^{2n} \left( {b - x}
\right)^{2n} dx}
\end{align*}
Therefore,
\begin{align}
\int_a^b {\left(f^{\left({2n}\right)} \left( x \right)\right)^2dx}
&\ge \frac{{2\alpha }}{{\left( {b - a} \right)^{2n}}} f\left( {x_0
} \right)\int_a^b {\left( {x - a} \right)^n\left( {b - x}
\right)^nf^{\left({2n}\right)}\left( x \right)dx}
\nonumber\\
&\qquad- \frac{{\alpha^2 }}{{\left( {b - a} \right)^{4n}
}}f^2\left( {x_0 } \right)\int_a^b {\left( {x - a} \right)^{2n}
\left( {b - x} \right)^{2n} dx}
\nonumber\\
&= \frac{{2\alpha }}{{\left( {b - a} \right)^{2n}}} f\left( {x_0 }
\right)\int_a^b {\left( {x - a} \right)^n\left( {b - x}
\right)^nf^{\left({2n}\right)}\left( x \right)dx}
\\
&\qquad-\alpha^2  \left( {b - a} \right)f^2\left( {x_0 } \right)
{\rm B}\left( {2n + 1,2n + 1} \right).\nonumber
\end{align}
It is not difficult to observe that
\begin{align*}
\int_a^b {\left( {x - a} \right)^n\left( {b - x}
\right)^nf^{\left( {2n} \right)}\left( x \right)dx}=0,
\end{align*}
which follows by integrating by parts and using the given
assumptions.

Now, by triangle inequality we have
\begin{align*}
\int_a^b {\left|f^{\left({2n}\right)} \left( x \right)\right|^2dx}
&\ge \left|\int_a^b {\left(f^{\left({2n}\right)} \left( x
\right)\right)^2dx}\right|
\\
&\ge \alpha^2  \left( {b - a} \right)\left|f\left( {x_0 }
\right)\right|^2 {\rm B}\left( {2n + 1,2n + 1} \right),\nonumber
\end{align*}
simple computations gives the required result (\ref{eq2.1}).
\end{proof}

\section{Useful Applications}

Let $f:I\subseteq \mathbb{R}\rightarrow \mathbb{R}$, be a twice
differentiable mapping such that $f^{\prime \prime }\left( x\right) $ exists on $%
I^{\circ}$, and $\left\| {f''} \right\|_\infty=\sup_{x\in \left(
a,b\right) }\left\vert {f^{\prime \prime }\left( x\right)
}\right\vert <\infty $. Then the trapezoid inequality
\begin{align}
\label{trapineq} \left| {\left( {b - a} \right)\frac{{f\left( a
\right) + f\left( b \right)}}{2}-\int_a^b {f\left( x \right)dx} }
\right| \le \frac{{\left( {b - a} \right)^3 }}{{12}}\left\| {f''}
\right\|_\infty,
\end{align}
holds. Therefore, the integral $\int_a^b {f\left( x \right)dx} $
can be approximated  in terms of the trapezoidal rules,
respectively such as:
\begin{eqnarray*}
\int_a^b {f\left( x \right)dx}  \cong \left( {b - a}
\right)\frac{{f\left( a \right) + f\left( b \right)}}{2}.
\end{eqnarray*}
By means of (\ref{eq2.10}), it is significant to remark that the
inequality has a trapezoid bound term, therefore  we may rewrite
(\ref{eq2.10}) to obtain a new upper bound for the trapezoid
inequality, such as:
\begin{corollary}
Under the assumptions of Theorem \ref{thm6}, we have
\begin{align}
\mathcal{T}_{\rm{rap}}\left({f}\right)\le \frac{b-a}{2}
\cdot\mathop {\max }\limits_{a \le x \le b} f\left( x \right)+
\frac{\left( b-a \right)^2}{24\cdot\mathop {\max }\limits_{a \le x
\le b} f\left( x \right) } \cdot \int_a^b {f'^2 \left( x
\right)dx},
\end{align}
provided that $\mathop {\max }\limits_{a \le x \le b} f\left( x
\right) \ne 0$. Equivalently, in terms of norms we may write
\begin{align}
\left|{\mathcal{T}_{\rm{rap}}\left({f}\right)}\right|\le
\frac{b-a}{2} \cdot \left\| f \right\|_\infty + \frac{\left( b-a
\right)^2}{24} \cdot \frac{\left\| f' \right\|^2_2}{\left\| f
\right\|_\infty  } ,
\end{align}
where; $\left\| f \right\|_\infty=\mathop {\sup }\limits_{a \le x
\le b} \left|f\left( x \right) \right| $ and $\left\| f'
\right\|^2_2=\int_a^b {\left|f'\left( x \right) \right|^2dx}$. The
two inequalities are sharp.
\end{corollary}
Henceforth, by setting
\begin{align*}
M:= \frac{6}{\left({b-a}\right)^2}\cdot \left\| f \right\|_\infty
+ \frac{1}{2\left({b-a}\right)} \cdot \frac{\left\| f'
\right\|^2_2}{\left\| f \right\|_\infty  },
\end{align*}
a beautiful trapezoid inequality may be written as:
\begin{align}
\label{eq2.14} \left| {\left( {b - a} \right)\frac{{f\left( a
\right) + f\left( b \right)}}{2}-\int_a^b {f\left( x \right)dx} }
\right| \le \frac{{\left( {b - a} \right)^3 }}{{12}}M,
\end{align}
which holds with more less restrictions on $f$, and so if $f$ is
twice differentiable and has   bounded second derivative, with
$M\le\left\| {f''} \right\|_\infty$, then totally (\ref{eq2.14})
by its assumptions  can be better than (\ref{trapineq}), and
exactly if $M:=\left\| {f''} \right\|_\infty$. So that we have
applied our result (\ref{eq2.10}) to obtain  new trapezoid type
inequality which has important applications in numerical
integrations.\\

One more direct interesting application is to bound the geometric
mean by a sharp upper bound. This happens if one assumes
$f(a)f(b)>0$, which already holds by assumptions of Theorem
\ref{thm3}, then (\ref{eq2.1})  can be written as:
\begin{align*}
\sqrt{f\left({a}\right)f\left({b}\right)} \le
\sqrt{\frac{b-a}{12}}\cdot \left({\int_a^b {f'^2 \left( x
\right)dx}}\right)^{1/2},
\end{align*}
equivalently we write,
\begin{align}
\label{G}G\left({f\left({a}\right),f\left({b}\right)}\right)\le
\sqrt{\frac{b-a}{12}}\cdot \left\| {f'} \right\|_2,
\end{align}
where $G\left({\cdot,\cdot}\right)$ is the geometric mean and the
inequality is sharp.

Moreover, if $f$ is log-convex, i,e., $f$ satisfies the inequality
\begin{align*}
f\left( {\lambda x + \left( {1 - \lambda } \right)y} \right) \le
f^\lambda  \left( x \right)f^{1 - \lambda } \left( y \right).
\end{align*}
for all $x, y \in [a, b]$ and $\lambda \in[0,1]$. In particular,
choose $\lambda=\frac{1}{2}$, then the double inequality
\begin{align}
f\left( {A\left( {x,y} \right)} \right) \le
G\left({f\left({x}\right),f\left({y}\right)}\right)\le
\sqrt{\frac{y-x}{12}}\cdot \left\| {f'} \right\|_2,\label{AG}
\end{align}
holds and sharp; provided that $a\le x < y \le b$, where
$A\left({\cdot,\cdot}\right)$ is the arithmetic mean. Clearly, the
left-hand side inequality sharp by the definition of
log-convexity.

A generalization of this result can be done if $f$ is considered
to be bijective on $[a,b]$. Choosing $\alpha,\beta \in [a,b]$ such
that $f \left( \alpha \right)=f^\lambda\left( x \right)$ and $f
\left( \beta \right)=f^{1-\lambda}\left( y \right)$, for some
$\lambda \in [0,1]$ and $x,y \in [a,b]$.  Making use of
(\ref{eq2.1}) we have
\begin{align}
\label{eq3.7}f\left( \alpha  \right)f\left( \beta  \right) \le
\frac{{\beta - \alpha }}{{12}}\int_\alpha ^\beta  {f'^2 \left( x
\right)dx}.
\end{align}
Therefore, a generalization of (\ref{AG}) may given as:
\begin{align}
f\left( {\lambda x + \left( {1 - \lambda } \right)y} \right) &\le
f^\lambda  \left( x \right)f^{1 - \lambda } \left( y \right)
\nonumber\\
&\le \frac{{f^{ - 1} \left( {f^{1 - \lambda } \left( y \right)}
\right) - f^{ - 1} \left( {f^\lambda \left( x \right)}
\right)}}{{12}} \cdot \int_{f^\lambda  \left( x \right)}^{f^{1 -
\lambda } \left( y \right)} {f'^2 \left( x \right)dx}.
\end{align}
or written in terms of generalized means, as
\begin{align*}
f\left( {A_\lambda  \left( {x,y} \right)} \right) \le G_\lambda
\left( {f\left( x \right),f\left( y \right)} \right) \le
\frac{{f^{ - 1} \left( {f^{1 - \lambda } \left( y \right)} \right)
- f^{ - 1} \left( {f^\lambda  \left( x \right)} \right)}}{{12}}
\cdot \int_{f^\lambda  \left( x \right)}^{f^{1 - \lambda } \left(
y \right)} {f'^2 \left( x \right)dx}
\end{align*}
where, $A_\lambda  \left( {x,y} \right)=\lambda x + \left( {1 -
\lambda } \right)y$, is the generalized arithmetic mean and
$G_\lambda  \left( {x,y} \right)= x^{\lambda} y^{1 - \lambda } $,
is  the generalized geometric mean. \centerline{}

\centerline{}

\end{document}